\documentclass[10pt]{amsart}

\usepackage{amsmath, amsfonts, amssymb, xypic}
\usepackage{latexsym, verbatim}
\usepackage{graphicx}
\usepackage{color}
\usepackage{url}
\usepackage{slashed}
\usepackage[table]{xcolor}
\usepackage[hidelinks]{hyperref}
\usepackage{amsmath}
\usepackage{enumitem}

\usepackage[textwidth=14cm,hcentering]{geometry}

\newtheorem{theo}{Theorem}[section]
\newtheorem{lemm}[theo]{Lemma}
\newtheorem{prop}[theo]{Proposition}
\newtheorem{coro}[theo]{Corollary}
\theoremstyle{definition}

\newtheorem{exam}[theo]{Example}

\theoremstyle{remark}
\newtheorem{rema}[theo]{Remark}

\numberwithin{equation}{section}

\newcommand{\Ker}{\mathrm{Ker}}
\newcommand{\Coker}{\mathrm{Coker}}
\newcommand{\Img}{\mathrm{Im}}

\newcommand{\ind}{\mathrm{ind}}

\newcommand{\lra}{\longrightarrow}

\newcommand{\CC}{\mathbb{C}}

\newcommand{\QQ}{\mathbb{Q}}
\newcommand{\RR}{\mathbb{R}}

\newcommand{\Z}{\mathbb{Z}}
\newcommand{\R}{\mathbb{R}}

\newcommand{\Aa}{\mathcal{A}}

\newcommand{\Hh}{\mathcal{H}}
\newcommand{\Ii}{\mathcal{I}}

\newcommand{\del}{\partial}
\newcommand{\delb}{{\bar \partial}}
\newcommand{\mub}{{\bar \mu}}
\newcommand{\dirac}{\slashed{\partial}}

\newcommand{\dR}{\mathrm{dR}}

\title{Hodge-de Rham numbers of almost complex 4-manifolds}

\author[J. Cirici]{Joana Cirici}
\address[J. Cirici]{Departament de Matemàtiques i Informàtica, Universitat de Barcelona\\
Gran Via 585\\
08007 Barcelona, Spain  / Centre de Recerca Matemàtica, Edifici C, Campus Bellaterra, 08193 Bellaterra, Spain}
\email{jcirici@ub.edu}

\author[S. Wilson]{Scott O. Wilson}
  \address[S. Wilson]{Department of Mathematics, Queens College, City University of New York, 65-30 Kissena Blvd., Flushing, NY 11367}
  \email{scott.wilson@qc.cuny.edu}

\thanks{
J. Cirici acknowledges the Serra H\'{u}nter Program. 
Her work was also partially funded by the Spanish State Research Agency (Mar\'{i}a de Maeztu Program CEX2020-001084-M and I+D+i project PID2020-117971GB-C22/MCIN/AEI/10.13039/501100011033) as well as by the
French National Research Agency (ANR-20-CE40-0016). 
S. O. Wilson acknowledges support provided by a PSC-CUNY Award, jointly funded by The Professional
Staff Congress and The City University of New York (TRADB  \#  63528-00 51).
}

\begin{document}

\maketitle

\begin{abstract}
We introduce and study Hodge-de Rham numbers for compact almost complex 4-manifolds, generalizing the Hodge numbers of a complex surface. The main properties of these numbers in the case of complex surfaces are extended to this more general setting, and it is shown that all Hodge-de Rham numbers for compact almost complex 4-manifolds are determined by the cohomology, except for one (the irregularity). Finally, these numbers are shown to prohibit the existence of complex structures on certain manifolds, without reference to the classification of surfaces.  
\end{abstract}

\tableofcontents

\section{Introduction}
Hodge numbers are fundamental analytic invariants of complex manifolds.
They arise as the dimensions $h^{p,q}:=\dim H^{p,q}_\delb$ of the Dolbeault cohomology groups 
which, by Dolbeault's Theorem, admit a sheaf-theoretic description using holomorphic differential forms.
In the compact case, Hodge theory gives an isomorphism between Dolbeault cohomology and the 
spaces of $\delb$-harmonic forms, defined after choosing a Hermitian metric.
This isomorphism gives the following \textit{Serre duality identities} for any compact complex manifold of complex dimension $n$:

\begin{enumerate}[label={(I.\arabic*)}, topsep=2pt,itemsep=4pt,partopsep=4pt, parsep=4pt]\setcounter{enumi}{0}
\item \label{ISerre} $h^{0,0}=h^{n,n}=1$ and $h^{p,q}=h^{n-p,n-q}$ for all $1\leq p,q\leq n$.
\end{enumerate}

Dolbeault cohomology is related to the topology of the manifold by means of a spectral sequence,
called the \textit{Fr\"{o}licher spectral sequence}.
One consequence of the convergence of this spectral sequence in the compact case is the \textit{Euler characteristic identity}:

\begin{enumerate}[label={(I.\arabic*)}, topsep=2pt,itemsep=4pt,partopsep=4pt, parsep=4pt]\setcounter{enumi}{1}
\item \label{Ieuler} $e:=\sum (-1)^k b^k=\sum (-1)^{p+q} h^{p,q},$
\end{enumerate}
where $b^k$ denotes the $k$-th Betti number of the manifold.

Together with plurigenera, Hodge numbers are the most important invariants in the Enriques-Kodaira classification of compact complex surfaces. Many other invariants are written as linear combinations of Hodge numbers, such as:
\begin{itemize}[topsep=2pt,itemsep=1pt,partopsep=4pt, parsep=4pt]
 \item (\textit{Irregularity}) $\mathfrak{q}:=h^{0,1}$.
 \item (\textit{Geometric genus}) $p_g:=h^{0,2}$.
  \item (\textit{Holomorphic Euler characteristic}) $\chi:=\sum (-1)^q h^{0,q}=1-\mathfrak{q}+p_g$.
\end{itemize}

The above invariants are defined in general for any compact complex manifold.
In the case of compact complex surfaces, which are complex manifolds of complex dimension 2, the following strong relations are satisfied:

\begin{enumerate}[label={(I.\arabic*)},topsep=2pt,itemsep=1pt,partopsep=4pt, parsep=4pt]\setcounter{enumi}{2}
\item \label{IBetti} (\textit{Degeneration}) $b^k=\sum_{p+q=k} h^{p,q}$.
 \item \label{Isignature}(\textit{Signature}) $\sigma=4\chi-e=\sum (-1)^q h^{p,q}$.
 \item \label{Inoether}(\textit{Noether's formula}) $12\chi=c_1^2+e$.
\end{enumerate}
Here $\sigma:=b^+-b^-$ denotes the signature of the manifold and $c_1$ is the first Chern class, 
whose square is a topological invariant.
The above relations altogether imply that Hodge numbers of compact complex surfaces are determined 
from the Betti numbers of the manifold and the signature.
In particular, they depend only on the oriented real cohomology ring of the manifold.

\medskip

The purpose of this note is to introduce and study a generalization of Hodge numbers that is valid for any almost complex 4-manifold.
The starting point is the Fr\"{o}licher-type spectral sequence introduced in \cite{CWDol} for any almost complex manifold and the corresponding Dolbeault cohomology groups arising as the first stage of this spectral sequence.
For a compact almost complex 4-manifold, the Dolbeault cohomology vector spaces 
introduced in \cite{CWDol} are always finite-dimensional for the bottom $(*,0)$ and top $(*,2)$ bidegrees, respectively, and are identified with certain spaces of harmonic forms.
However, in bidegree $(0,1)$ such spaces have no supporting Hodge theory and, as shown in \cite{CPS},
are infinite dimensional in the non-integrable case.
Still, the Fr\"{o}licher spectral sequence always degenerates by the second page $E_2$, at the latest.
It therefore makes sense to consider the \textit{Hodge-de Rham numbers} 
\[h^{p,q}:=\dim E_2^{p,q}=\dim E_\infty^{p,q},\]
which are analytic invariants of the manifold, and always finite in the compact case. Of course, these numbers trivially satisfy
Equations \ref{Ieuler} and \ref{IBetti}, and reduce to the classical Hodge numbers in the case of a complex surface.
We show that the Serre identities \ref{ISerre} are also satisfied, and
provide relations generalizing Equations \ref{Isignature} and \ref{Inoether}.
 
The presentation below will not assume any previous knowledge on complex surfaces, but rather 
we will prove from first principles the general properties for almost complex 4-manifolds, and observe how these
collapse to the well-known properties mentioned above in the case of complex surfaces. Additionally, we prove special properties for non-integrable structures.
This exercise allows us to describe with precision which of the properties are special, and only satisfied in the case of complex surfaces, and which of the properties still hold in the non-integrable case, even if their implications are not as dramatic as in the integrable setting. As one somewhat surprising result of this, we show that for non-integrable structures, all Hodge-de Rham numbers may be computed from the Betti numbers together with the irregularity $\mathfrak{q}:=h^{0,1}$. Moreover, this number is lower semi-continuous under small deformations and is bounded by $b^1\leq 2\mathfrak{q}\leq 2b^1$. In particular, when $b^1\leq 1$, the irregularity and hence all Hodge-de Rham numbers become topological invariants.

There are many known examples of almost complex 4-manifolds not admitting any integrable almost complex structure. The arguments for proving that such example exist often rely on Kodaira's classification of compact complex surfaces. Using Hodge-de Rham numbers we show how, in many situations, one can prove such results without invoking Kodaira's classification.

Another natural approach for generalizing Hodge numbers to the non-integrable setting arises after choosing
a Hermitian metric and considering the spaces of $\delb$-harmonic forms $\Hh_\delb^{p,q}:=\Ker(\Delta_\delb)|_{p,q}$, as already noted by Hirzebruch in \cite{Hirzebruch}. For compact almost complex manifolds, these spaces are finite-dimensional by elliptic operator theory, even in the non-integrable case, and their dimensions $h^{p,q}_\delb:=\dim \Hh_\delb^{p,q}$ 
satisfy the Serre duality identities. We have $h^{0,0}_\delb=h^{0,0}$ as well as inequalities 
$h^{p,0}_\delb\geq h^{p,0}$ for $p=1,2$, which give the same relations for the Serre-dual numbers. However, there is no obvious relation between the numbers $h^{p,1}_\delb$ and $h^{p,1}$. In fact, as shown by Holt and Zhang in \cite{HZ}, the number $h^{0,1}_\delb$ depends on the metric in general and may be unbounded as the metric varies, while $h^{0,1}$ is metric-independent by construction. 
Additionally, the number $h^{1,1}_\delb$ also depends on the metric in general, \cite{TaTo4}, 
but attains only two possible values \cite{Holt}.

This note is organized as follows.
In Section \ref{SecPrel} we review some preliminaries on the algebra of differential forms of an almost complex manifold. We prove some first basic properties of such differential forms, and review the Fr\"{o}licher spectral sequence of almost complex manifolds.
In Section \ref{SecHodge} we introduce Hodge-de Rham numbers of almost complex 4-manifolds and provide generalizations of the identitites \ref{ISerre}-\ref{Inoether}. In Section \ref{secinte} we study separately the integrable and non-integrable cases and establish a criterion of integrability in terms of Hodge theory. The last section is devoted to examples.

\subsection*{Acknowledgments}
We would like to thank Michael Albanese for sharing a draft of his appendix in \cite{DG}, as well as illuminating discussions.
We also thank Claude LeBrun for pointing us to the reference \cite{Aro}, and the referee for his useful comments.

\section{Preliminaries} \label{SecPrel}

\subsection{Differential forms of almost complex manifolds}
An \textit{almost complex structure} on a smooth manifold $M$ is given by an endomorphism  of the tangent bundle $J:TM\to TM$ such that $J^2=-\mathrm{Id}$. The pair $(M,J)$ is called an \textit{almost complex manifold}.
The extension of $J$ to the complexiﬁed
tangent bundle induces a bigrading on the complex de Rham algebra of differential forms
\[\Aa^{n}_{\dR}\otimes\CC=\bigoplus_{p+q=n}\Aa^{p,q}\]
and the exterior differential decomposes as
\[d=  \mub+\delb+\del+\mu\]
where the bidegrees of each component are given by 
\[
|\mub|=(-1,2), \,  |\delb|=(0,1), \, |\del|=(1,0), \text{ and } \,  |\mu|=(2,-1).
\]
Here $\mub$ and $\delb$ are complex conjugate to $\mu$ and $\del$ respectively.
The almost complex structure $J$ is said to be \textit{integrable} if and only if $\mu\equiv \mub\equiv 0$. In this case, by the Newlander-Nirenberg Theorem, the endomorphism $J$ is induced from a holomorphic atlas and so $M$ is a complex manifold.

The equation $d^2=0$ gives the following set of relations:
\begin{align*}
\mu^2 = 0 \\
\mu \del + \del \mu =  0 \\
\mu \delb + \delb \mu +\del^2 =  0 \\
\mu \mub + \del \delb + \delb \del + \mub \mu  = 0 \\ 
\mub \del + \del \mub +\delb^2 =  0 \\
\mub \delb + \delb \mub =  0 \\
\mub^2 = 0
\end{align*}
In the integrable case, for which $d=\delb+\del$, these equations collapse to the well-known equations of a double complex
$\del^2=0$, $\del\delb+\delb\del=0$ and $\delb^2=0$.

The following results are extensions, to the possibly non-integrable case, of well-known results for compact complex surfaces.
We first recall a result on pluriharmonic functions which relies on the Hopf maximum principle
(see Corollary 1 of \cite{CWAH} for a proof in the possibly non-integrable case).

\begin{lemm}\label{constant}Let $f$ be a smooth function on a compact connected almost complex manifold such that $\delb\del f=0$. Then $f$ is constant. In particular, $f$ is constant if $\delb f =0$.
\end{lemm}

A basic result for compact complex surfaces states that every holomorphic form is closed (see Lemma IV.2.1 of \cite{BaHu}). More generally, we have:
\begin{lemm}\label{holomclosed}
 Let $\alpha\in \Aa^{p,0}$ be a differential form on a compact almost complex 4-manifold
 such that $\delb\alpha=\mub\alpha=0$. Then $d\alpha=0$.
\end{lemm}
\begin{proof}
 The case $p=0$ follows from Lemma \ref{constant} and the case $p=2$ is trivial for bidegree reasons. The case $p=1$ is an easy consequence of Stokes Theorem (c.f. Lemma 3.13 in \cite{CWDol}). Indeed, by assumption we have 
  $d \alpha = \del \alpha$ is a $(2,0)$-form, while on the other hand
 \[\int d\alpha\wedge d\overline{\alpha}=\int d(\alpha\wedge d\overline \alpha)=0.\]
 Since the pairing $(\alpha,\beta) = \int \alpha \wedge \bar \beta$ defines a positive definite inner product on $(2,0)$-forms, 
 we deduce $\del \alpha=0$, and so $\alpha$ is $d$-closed. 
\end{proof}

The above result is actually valid for forms of type $(n,0)$ or $(n-1,0)$ on any compact $2n$-dimensional 
almost complex manifold, but not in general for $(p,0)$-forms.
Lemma IV.2.3 in \cite{BaHu} extends now easily to non-integrable structures:

\begin{lemm}\label{vanishes}Let $\alpha\in \Aa^{p,0}$ be a differential form on a compact almost complex 4-manifold
 such that $\delb\alpha=\mub\alpha=0$. If $\alpha=\del \beta$ then $\alpha=0$.
\end{lemm}
\begin{proof}
 Assume first that $\alpha\in \Aa^{1,0}$. If $\alpha=\del \beta$ then $\delb\del \beta=0$ and so $\alpha=0$ by Lemma \ref{constant}.
 If $\alpha\in \Aa^{2,0}$ and $\alpha=\del \beta$, we have
 \[\int \alpha\wedge \overline{\alpha}=0\]
 and arguing as in the proof of Lemma \ref{holomclosed} we find $\alpha=0$.
\end{proof}

\subsection{Fr\"{o}licher spectral sequence}
The \textit{Hodge filtration} of a complex manifold is given by the column filtration 
\[F^p\Aa^n:=\bigoplus_{q\geq p} \Aa^{q,n-q}.\]
In the non-integrable case, this filtration is not compatible with the exterior differential.
However, there is an elementary way to modify the usual Hodge filtration making it compatible with the exterior differential for all almost complex manifolds, and reducing to the Hodge filtration in the integrable case \cite{CWDol}. Namely, we simply restrict to those forms in the first column which are in the kernel of $\mub$, i.e. 
\[F^p\Aa^n:=\Ker(\mub)\cap \Aa^{p,n-p}\oplus\bigoplus_{q>p} \Aa^{q,n-q}.\]
One then verifies that $d(F^p\Aa^n)\subseteq F^p\Aa^{n+1}$ and so $(\Aa^*,d,F)$ is a filtered complex.
The first stage of its associated spectral sequence may be written as the quotient
\[E_1^{p,q}\cong {\{x\in \Aa^{p,q}; \mub x=0, \delb x=\mub y\}\over\{x=\mub a +\delb b; \mub b=0\}}
\]
with differential 
\[\delta_1[x]=[\del x-\delb y].\]
The second stage is given by the quotient
\[E_2^{p,q}\cong {\{{x\in \Aa^{p,q}; \mub x=0, \delb x=\mub y, \del x=\delb y+\mub z\}}\over{\{x=\mub a+\delb b+\del c; 0=\mub b+\delb c, 0=\mub c\}}}.\]
Further terms of this spectral sequence are described in the Appendix of \cite{CWDol}. For the study of 4-dimensional manifolds it suffices to understand the terms $E_1$ and $E_2$, as we will soon see.
This spectral sequence converges to complex de Rham cohomology: 
\[H_{\dR}^n\otimes \CC\cong \bigoplus_{p+q=n}E_\infty^{p,q}.\]
We have the following degeneration result:

\begin{lemm}\label{degeneration}On any compact almost complex 4-manifold we have:
\begin{enumerate}
 \item $E_1^{p,q}=E_\infty^{p,q}$ for all $0\leq p\leq 2$ and $q\in \{0,2\}$.
 \item $E_2^{p,1}=E_\infty^{p,1}$ for all $0\leq p\leq 2$.
\end{enumerate}
\end{lemm}
\begin{proof}
Lemma \ref{constant} implies  $E_1^{0,0}=E_\infty^{0,0}$ (c.f. Corollary 4.9 of \cite{CWDol})
 and Lemma 3.13 of \cite{CWDol} gives $E_1^{2,0}=E_\infty^{2,0}$ using the non-degenerate pairing in
the middle degree, as in the proof of Lemma \ref{vanishes} above. 
These two results imply  $E_1^{1,0}=E_\infty^{1,0}$ since any non-trivial 
differential $\delta_1$ with source or origin at $E_1^{1,0}$ would affect non-trivially 
one of the vector spaces $E_1^{0,0}$ or $E_1^{2,0}$. 
By Proposition 4.10 of \cite{CWDol}, we have 
\[E_1^{p,0} = \Ker \left( \Delta_\delb + \Delta_\mub \right) \cap \Aa^{p,0}\text{ and } 
E_1^{p,2} = \Ker \left( \Delta_\delb + \Delta_\mub \right) \cap \Aa^{p,2}.\] These are isomorphic under the Serre-Duality map
$\bar \star: E_1^{p,0} \cong E_1^{2-p,2}$. Moreover, this isomorphism intertwines the differentials $\del: E_1^{p,0} \to E_1^{p+1,0}$
and $\del^*: E_1^{p,2} \to E_1^{p-1,2}$. Therefore we have that $\del^* = 0$ on $E_1^{*,2}$.
This proves  $E_1$-degeneration
along the top row $q=2$ as well.
The second assertion now follows since the differential $\delta_r$ with $r \geq2 $ has bidegree $(r,r-1)$, so that $r-1>0$ and 
the differential must vanish by the first assertion.
\end{proof}

\section{Hodge-de Rham numbers}\label{SecHodge}

\subsection{Degeneration, Euler and Serre identities}
Define the \textit{Hodge-de Rham numbers} of a compact almost complex 4-manifold as
\[h^{p,q}:=\dim E_2^{p,q}.\]
Lemma \ref{degeneration} gives the degeneration identity
\[b^k=\sum_{p+q=k}h^{p,q}\]
of Equation \ref{IBetti},
and in particular one obtains the Euler characteristic in terms of Hodge-de Rham numbers
\[e=\sum (-1)^{p+q} h^{p,q},
\]
thus recovering Equation \ref{Ieuler}.
The following result restricts the possible Hodge-de Rham numbers in total degree 1.

\begin{lemm}\label{h01h10}
On a compact almost complex 4-manifold we have \[h^{1,0}\leq h^{0,1}\text{ and }b^1\leq 2h^{0,1}.\]
\end{lemm}
\begin{proof}
The numbers $h^{0,1}$ and $h^{1,0}$ are, by definition, the dimensions of the vector spaces 
\[E_2^{0,1}\cong{\{\alpha\in \Aa^{0,1}; \delb\alpha=\mub \beta, \del \alpha=\delb \beta\}\over\{\alpha=\delb f\}}, \quad{\text{and }}\quad E_2^{1,0}\cong \Aa^{1,0}\cap \Ker(d),\]
respectively.
Define a map $E_2^{1,0}\lra E_2^{0,1}$ by letting $\alpha\mapsto [\overline{\alpha}]$
This is well-defined, since $d\alpha=0$. We show that it is injective.
If $[\overline \alpha]=0$ then $\overline\alpha=\delb f$
and so $\alpha=\del f$ which gives $\delb \del f=0$. Then $f$ is constant by Lemma \ref{constant}
and so $\alpha=0$. This proves $h^{1,0}\leq h^{0,1}$. Since $b^1=h^{1,0}+h^{0,1}$, we get the inequality $b^1\leq 2h^{0,1}$.
\end{proof}

Serre duality is also satisfied for these numbers: 
\begin{prop} \label{SerreDuality}
On any compact almost complex 4-manifold we have $h^{p,q}=h^{2-p,2-q}$.
\end{prop}
\begin{proof}
The identities $h^{*,0}=h^{2-*,2}$ follow from Serre duality of Dolbeault 
cohomology proven in Corollary 4.11 of \cite{CWDol} together with the fact that $E_1=E_2$ in such bidegrees, by Lemma \ref{degeneration} above.
Finally, Poincare duality gives an equality 
\[b^1=h^{0,1}+h^{1,0}=h^{1,2}+h^{2,1}=b^3,\]
and since $h^{1,0}=h^{1,2}$, this implies the remaining identity $h^{0,1}=h^{2,1}$.
\end{proof}

\subsection{Signature formula}
In this section we generalize Equation \ref{Isignature} to the non-integrable setting. Though we only present here the dimension $4$ case, there is a result in all dimensions due to Hirzebruch \cite{H}, with a proof given recently by Albanese in an appendix to \cite{DG}. Let 
\[\widetilde\sigma:=\sum_{p,q=0}^2 (-1)^q h^{p,q}\]
and consider the \textit{holomorphic Euler characteristic}
\[\chi:= \sum_{q=0}^2 (-1)^q h^{0,q}.\]
We immediately obtain the following relation:
\begin{coro}\label{tildesigma}
On any compact almost complex 4-manifold we have $\widetilde\sigma=4\chi-e$.
\end{coro}
\begin{proof}
Using Serre duality we may write:
\begin{align*}
e & = \sum_{p,q=0}^2 (-1)^{p+q} h^{p,q} =\\
& =  - \left( \sum_{p,q=0}^2 (-1)^q h^{p,q} \right) + 2 \left( \sum_{q=0}^2 (-1)^q h^{0,q} +   \sum_{q=0}^2 (-1)^q h^{2,q} \right) =\\
& = - \left( \sum_{p,q=0}^2 (-1)^q h^{p,q} \right) + 4 \left( \sum_{q=0}^2 (-1)^q h^{0,q} \right)=-\widetilde\sigma+4\chi.\qedhere
\end{align*}
\end{proof}

For any compact almost complex 4-manifold, Hirzebruch's Signature Theorem gives the relation
\[\sigma = \frac 1 3 p_1 = \frac 1 3 ( c_1^2-2e),\]
where $p_1$ is the first Pontryagin class (see \cite{Hirbook}).
Also, recall that the \textit{top Todd class} of any compact 4-manifold is given by 
\[Td={1\over 12}(c_1^2+e).\]
Combining the above two identities we have 
\[e + \sigma = 4\cdot Td.\]
In particular we have $\sigma \equiv -e \, (\mathrm{mod}\, 4)$ for any almost complex $4$-manifold. 
This implies, for instance, that the connected sum of two almost complex manifolds is never almost 
complex, e.g. $S^4$ is not almost complex, since it is the unit for connected sums. 
Another consequence is the following:

\begin{prop}\label{sigmas}
 On any compact almost complex 4-manifold we have $\sigma-\widetilde\sigma=4(Td-\chi)$.
\end{prop}
\begin{proof} 
It follows from the identity
$e + \sigma = 4\cdot Td$ together with  Corollary \ref{tildesigma}.
\end{proof}

The above result gives $\sigma\equiv \widetilde\sigma\, (\mathrm{mod}\, 4)$. As we will see in Section \ref{secinte}, for the case of complex surfaces, the index theorem implies that $\chi=Td$ so that $\sigma = \widetilde\sigma$, for a complex surface. 
This does not hold in general for almost complex manifolds as we will see in the examples of  Section \ref{SecExamples}.

\subsection{Noether's formula}

For any almost complex manifold there is an associated spin$^c$ structure and Dirac operator 
\[\dirac^c: \Gamma(\mathbb{S^+}) \to
\Gamma(\mathbb{S}^-),\] where 
$\mathbb{S}^+$ may be identified with $\Lambda^{0,\textrm{even}}$ and 
$\mathbb{S}^-$ may be identified with $\Lambda^{0,\textrm{odd}}$. 
According to Gauduchon \cite{Gau},
on a compact almost Hermitian $4$-manifold, and for any connection $\nabla$, the operator $\dirac^c$ is given by 
 \[
 \dirac^c \phi = \sqrt 2 \left( \delb + \delb^* \right) + \frac 1 4 \theta . \epsilon(\phi) +  \frac 1 2 i  a. \phi.
 \]
 Here $\theta$ is the Lee form characterized by $d \omega = \theta \wedge \omega$, 
 $ia = \nabla - \nabla^{\textrm{chern}}$ is the $1$-form that measures the difference from the Chern connection, 
 $\epsilon$ is the parity operator, and the dot indicates Clifford multiplication. 
 The operator $\delb^*$ is the formal adjoint to $\delb$ with respect to the chosen Hermitian metric.
 In particular, modulo zero order terms, $ \dirac^c \phi$ is equal to $ \sqrt 2\left( \delb + \delb^* \right)$. As a consequence, these operators have the same index.
 The Atiyah-Singer index theorem for $\dirac^c$ (see for instance \cite{Muk}) gives
\[
\ind( \dirac^c) = \ind ( \delb + \delb^* )=Td,
\]
where $Td$ is the top Todd class. This gives a natural generalization of Equation \ref{Inoether}:

\begin{prop}\label{NoetherNonInt}On any compact almost hermitian 4-manifold, the vector spaces
\[\Ker(\delb + \delb^*)\text{ and }\Coker(\delb + \delb^*)\] are finite dimensional and
 \[\dim \Ker(\delb + \delb^*) - \dim \Coker(\delb + \delb^*)={1\over 12}(c_1^2+e).\]
\end{prop}

The right hand side is integral, and the left hand side is purely topological (since $c_1^2$ is topological, by Hirzebruch's Signature Theorem). In particular, as already noted by Hirzebruch \cite{Hirbook} and Van de Ven \cite{VdV}, we obtain:

\begin{coro}
On any compact almost complex 4-manifold $c_1^2+e\equiv 0 \,({mod}\, 12).$
\end{coro}

Conversely, in \cite{VdV} it is shown that for any pair of integers $(p,q)$ with $p+q\equiv 0 \,(\mathrm{mod}\, 12)$, there is a compact almost complex 4-manifold with $p=c_1^2$ and $q=e$.

There is a bit more one can say from the above proposition.  Note that if $\delb^2\neq 0$, the spaces $\Img(\delb)$ and $\Img(\delb^*)$ need not be orthogonal, so that  $\Ker(\delb + \delb^*)$  may not split as subspaces of bidegrees $(0,0)$ and $(0,2)$.
Nevertheless, we have the following.
Recall that 
\[\Hh^{p,q}_\delb:=\Ker(\Delta_\delb) \cap \Aa^{p,q} \quad \text{and} \quad h^{p,q}_\delb:=\dim \Hh_\delb^{p,q}.\]
The spaces $\Hh_\delb^{p,q}$ are always finite-dimensional in the compact case and we have Serre-duality isomorphisms, giving identities $h^{p,q}_\delb=h^{m-p,m-q}_\delb$ for any almost complex manifold of dimension $2m$.
We have:

\begin{coro}
For any compact almost complex $4$-manifold
\[
\dim \Coker(\delb + \delb^*) = h_\delb^{0,1}\quad\text{ and }\quad
h^{0,1}_\delb \geq 1+h^{2,0}_\delb - {1\over 12}(c_1^2+e).
\]
\end{coro}

We note in the inequality, the left hand side is metric dependent \cite{HZ}, who have shown $h^{0,1}_\delb$ can be arbitrarily large. But it is bounded below by the right hand side, which depends only on the topology and the almost complex structure, since $h^{2,0}_\delb$ is metric independent. 

\begin{proof}
Consider the map 
$\Hh_\delb^{0,1}\to \Coker(\delb+\delb^*)$
given by $x\mapsto [x]$. It is surjective since given $[x]$ we may choose a representative $x$ and take its decomposition 
$x=\Hh_\delb(x)+\delb\delb^*Gx+\delb^*\delb Gx$, where $G$ is Green's operator for $\Delta_\delb$ and $\Hh_\delb(x)$ denotes the projection to $\delb$-harmonic forms. 
Also, this map is injective since $\Hh_\delb^{0,1}$ is orthogonal to both images $\Img(\delb)$ and $\Img(\delb^*)$. 

Proposition \ref{NoetherNonInt} then yields:
\[
h^{0,0}_\delb - h^{0,1}_\delb+h^{0,2}_\delb \leq  
\dim \Ker(\delb + \delb^*) - \dim \Coker(\delb + \delb^*)={1\over 12}(c_1^2+e)
\]
and we use the fact that $h^{0,2}_\delb = h^{2,0}_\delb$ by Serre Duality.
\end{proof}

\section{Integrability}\label{secinte}
In this section we show how, in the integrable case, the results obtained in the previous sections collapse to give Equations
\ref{Isignature} and \ref{Inoether}, recovering the well-known results on the Hodge numbers for compact complex surfaces.
We also prove particular results that are valid only in the non-integrable case, and discuss obstructions to integrability related to Hodge theory, for almost complex 4-manifolds.

\subsection{Complex surfaces} Let us now restrict to the integrable setting.
We first recall an inequality for the geometric genus $p_g:=h^{0,2}$ (see Lemma IV.2.6 of \cite{BaHu}).
\begin{lemm}\label{geomgenus}
On any compact complex surface we have $b^+\geq 2p_g$.
\end{lemm}
\begin{proof}
Consider the map
$g: E_1^{2,0}\cong \Ker (\Delta_\delb) \cap \Aa^{2,0} \to \Aa^{2,0} \oplus \Aa^{0,2}$ defined by $g(\alpha):=\alpha + \bar \alpha$.
This map is an injection, since $ \Aa^{2,0} \cap \Aa^{0,2} =0$. Also, it has image in the real $d$-harmonic forms, since
every closed $(2,0)$-form is harmonic.
The pairing on complex valued $2$-forms given by 
\[
(\alpha, \beta) = \int \alpha \wedge \bar \beta
\]
 is positive definite on $\Aa^{2,0} \oplus \Aa^{0,2}$, and restricts to the intersection pairing on real forms. So, the pairing is still positive definite when restricted to the space of real $d$-harmonic forms in $ \Aa^{2,0} \oplus \Aa^{0,2}$, and we have
 \[
 b^+ \geq \dim_{\R} \left( \Img \, g \right) = \dim_\R \left( E_1^{2,0} \right) = 2 h^{2,0}=2p_g\qedhere
\]
\end{proof}

We may now prove the main results on the Hodge numbers for compact complex surfaces  (see for instance Theorem IV.2.7 in \cite{BaHu}).

\begin{theo}\label{complexfinale}For any compact complex surface, the Fr\"{o}licher spectral sequence degenerates at $E_1$.
The identities $\sigma=4\chi-e$ and $12\chi=c_1^2+e$ are satisfied and:
\begin{enumerate}
 \item [(i)] If $b^1$ is even then $h^{0,1}=h^{1,0}$ and $h^{1,1}=b^-+1$.
 \item [(ii)] If $b^1$ is odd then $h^{0,1}=h^{1,0}+1$ and $h^{1,1}=b^-$.
\end{enumerate}
\end{theo}
\begin{proof} 
Hodge theory gives isomorphisms 
$\Hh_\delb^{p,q}\cong E_1^{p,q}$
so in particular, $h^{p,q}\leq h^{p,q}_\delb$.
Note as well that by Lemma \ref{degeneration} we have
$h_\delb^{p,0}=h^{p,0}$ and $h_\delb^{p,2}=h^{p,2}$.
A computation, using integrability $\delb^2=0$, shows that
\[\dim \Ker(\delb + \delb^*)=h^{0,0}+h^{0,2}=1+h^{0,2}\text{ and }
\dim \Coker(\delb + \delb^*)=h^{0,1}_\delb.\] Then,
Proposition \ref{NoetherNonInt} gives 
\[1-h^{0,1}_\delb+h^{0,2}=\frac{1}{12} \left(c_1^2 + e \right).\]
Combined with Hirzebruch's signature formula
$\sigma = \frac 1 3  \left( c_1^2 - 2 e\right)$
we get
\[(b^+-2h^{0,2})+(2h^{0,1}_\delb-b^1)=1.\]
Both terms within brackets are non-negative. Indeed,
by Lemma \ref{geomgenus} we have $b^+\geq 2h^{0,2}$ and
by Lemma \ref{h01h10} we have $b^1\leq 2h^{0,1}\leq 2h_\delb^{0,1}$.
We are left
with only two possibilities:
\begin{enumerate}
 \item [(i)]$b^+=2h^{0,2}+1$ and $b^1=2h^{0,1}_\delb$, or
 \item [(ii)] $b^+=2h^{0,2}$ and $b^1=2h^{0,1}_\delb-1$.
\end{enumerate}
Since $b^1=h^{0,1}+h^{1,0}\leq 2h^{0,1}\leq 2h_\delb^{0,1}$, both cases imply
$h^{0,1}=h^{0,1}_\delb$. This proves that $\chi=Td$ and so Corollary \ref{tildesigma} gives the signature formula
$\sigma=4\chi-e$. Combining it with Hirzebruch's Signature Theorem $\sigma={1\over 3}(c_1^2-2e)$, we get Noether's formula
$12\chi=c_1^2+e$.
We now prove degeneration at $E_1$.
By Serre duality we get $h^{2,1}=h^{2,1}_\delb$.
Therefore we have $E_1^{0,1}=E_2^{0,1}$ and $E_1^{2,1}=E_2^{2,1}$, which implies that both differentials
$E_1^{0,1}\stackrel{\delta_1}{\lra}E_1^{1,1}\stackrel{\delta_1}{\lra}E_1^{2,1}$ in the spectral sequence
are trivial and so $E_2^{1,1}=E_1^{1,1}$. This gives 
 $h^{1,1}=h^{1,1}_\delb$. Lastly, Lemma \ref{degeneration} gives $h^{p,q}=h^{p,q}_\delb$ when $q\in \{0,2\}$.
Therefore we have $E_1=E_\infty$ and the two possibilities above correspond to even and odd $b^1$ respectively.
\end{proof}

The two cases above correspond to K\"ahler and non-K\"ahler surfaces respectively. Note that the above result, 
together with Equation \ref{IBetti}
as well as the Serre duality identities \ref{ISerre} allows to determine all Hodge numbers of a complex surface from
the Betti numbers together with the signature of the intersection pairing.

\subsection{Non-integrable structures} 
In the non-integrable case, and for the
middle total degree, Hodge-de Rham numbers are completely determined as follows:
\begin{lemm}\label{noninth20}
On a compact non-integrable almost complex $4$-manifold we have $h^{1,1}=b^2$ and  $h^{0,2}=h^{2,0}=0$. 
\end{lemm}
\begin{proof}
By Proposition \ref{SerreDuality} we have $h^{0,2}=h^{2,0}$. Since $b^2=h^{2,0}+h^{1,1}+h^{0,2}$  it suffices to show $h^{2,0}=0$.
The argument is analogous to that of Lemma 5.6 in \cite{AK}, see also Lemma 2.12 in \cite{DrLiZh}.
By Proposition 4.10 of \cite{CWDol} every class $[\alpha]\in E_1^{2,0}$ is represented by a $d$-closed form $\alpha$, so in particular 
$\mub \alpha = 0$. Since the almost complex structure is non-integrable,
there is an open set $U$ on which $\mub: T^{2,0} \to T^{1,2}$ is pointwise nonzero and therefore $\alpha$ is zero on $U$. 
On the other hand, a $d$-closed $(2,0)$-form $\alpha$ is always harmonic.
Hence $\alpha$ is zero on an open set it is zero everywhere
by the continuation theorem of \cite{Aro} (see also Corollary 1 of  \cite{bar}).
\end{proof}

Together with Serre duality, the above result implies that, in the non-integrable case, all numbers $h^{*,*}$ may be deduced from the Betti numbers of the manifold together with the \textit{irregularity} $\mathfrak{q}:=h^{0,1}$, which depends on the almost complex structure.
Note that by Lemma \ref{h01h10}, $\mathfrak{q}$ is bounded by
$b^1\leq 2\mathfrak{q}\leq 2b^1$.
This immediately gives:
\begin{coro}
If $b^1\leq 1$ then the irregularity $\mathfrak{q}=b^1$, as well as all Hodge numbers, are  topological invariants. 
\end{coro}

In general, we have:

\begin{prop}\label{semic}
On any compact almost complex 4-manifold, the irregularity $\mathfrak{q}$ is lower semi-continuous under small deformations of the underlying almost complex structure.
\end{prop}
\begin{proof}
Since $\mathfrak{q}=b^1-h^{1,0}$ it suffices to prove that $h^{1,0}$ is upper semi-continuous.
This follows from the fact that 
 \[h^{1,0}=\dim \left( \Ker(\delb)\cap \Ker(\mub)|_{(1,0)} \right) = \dim \left( \Ker\left(\Delta_{\delb} +\Delta_{\mub} |_{(1,0)} \right) \right),\]
 and that in bidegree $(1,0)$, 
 $\Delta_{\delb} +\Delta_{\mub}$ is elliptic, with the same symbol as $\Delta_\delb$.
\end{proof}

In particular, since $\mathfrak{q}\leq b^1$, we have:

\begin{coro}\label{cntsmalldef}
If $\mathfrak{q}=b^1$ (and so $h^{1,0}=0$) then the irregularity and all Hodge-de Rham numbers remain constant under small deformations.
\end{coro}

\subsection{Obstructions to integrablility}

There are many known examples of almost complex 4-manifolds not admitting any integrable almost complex structure. The arguments for proving that such example exist often rely on Kodaira's classification of compact complex surfaces. We explain in this section how, in many situations, one can prove such results without invoking Kodaira's classification.

The Hodge-type filtration on a compact almost complex manifold induces a filtration of the complex 
de Rham cohomology $H^*:=H^*_{\dR}\otimes \CC$, given by:
\[F^1H^1:=\{[\alpha]; d\alpha=0, \alpha\in \Aa^{1,0}\}\subseteq F^0H^1:=H^1.\]
\[F^2H^2:=\{[\alpha]; d\alpha=0, \alpha\in \Aa^{2,0}\}\subseteq F^1H^2:=\{[\alpha]; d\alpha=0, \alpha\in \Aa^{1,1}\oplus\Aa^{2,0}\}\subseteq F^0H^2:=H^2.
\]
The filtration $F^*H^n$ is said to define a \textit{pure Hodge decomposition of weight $n$} if and only if
\[F^pH^n\oplus \overline{F}^qH^n\cong H^n\text{ for all }p+q-1=n.\]

\begin{prop}\label{Hodgedec}On any compact almost complex 4-manifold, $F$
induces a pure Hodge decomposition of weight 2 on $H^2$.
If moreover $h^{0,1}=h^{1,0}$, then $F$ also induces a pure Hodge decomposition of weight 1 on $H^1$.
\end{prop}
\begin{proof}In the integrable case, this is Proposition IV.2.9 in \cite{BaHu}. Assume that the structure is non-integrable.
By Lemma \ref{noninth20} we have $F^2H^2=0$ and $H^2\cong  F^1H^2$ and so trivially we have
\[H^2\cong  F^1H^2\oplus \overline{F}^2H^2.\]
Assume now that $h^{0,1}=h^{1,0}$.
Since $h^{1,0}=\dim F^1H^1$ we have $H^1\cong F^1H^1+ \overline{F}^1H^1$. We next prove that the intersection is trivial.
Any class in $ F^1H^1\cap \overline{F}^1H^1$ can be represented by $\alpha_0\in \Aa^{0,1}$ and also by a form $\alpha_1\in \Aa^{1,0}$ with $d\alpha_0=d\alpha_1=0$ and $\alpha_0-\alpha_1=df$. 
Then $\alpha_0=\alpha_1=0$ by Lemma \ref{vanishes} and the fact that $\Aa^{0,1} \cap \Aa^{1,0} =0$.
\end{proof}

Pure Hodge structures are intimately related to formality. Recall that a differential graded algebra $(\Aa,d)$ is said to be \textit{formal} if it is connected to its cohomology $H^*(\Aa)$ by a string of quasi-isomorphisms of differential graded algebras, where $H^*(\Aa)$ carries the trivial differential. A topological space $X$ is then called \textit{formal} if its rational algebra of piece-wise forms $\Aa_{pl}(X)$ is a formal differential graded algebra. Note that by the descent of formality from $\CC$ to $\QQ$, in order to prove that a compact complex manifold is formal, it suffices to prove that its complex de Rham algebra is a formal differential graded algebra.

In \cite{DGMS} there are two completely different proofs for the formality of K\"{a}hler manifolds. The first one and most well-known uses the $\del\delb$-lemma. The second proof, called ``the principle of two types'' uses the fact that, for K\"ahler manifolds, we have 
degeneration at $E_1$ together with pure Hodge decompositions in cohomology. For a complex manifold, 
satisfying the $\del\delb$-lemma is equivalent to having pure Hodge decompositions together with degeneration at $E_1$. Note, however,
that while the $\del\delb$-lemma only makes sense in the the case of complex manifolds, Hodge decompositions and degeneration conditions
make sense in the abstract broader context of a commutative dg-algebra $\Aa$ defined over $\mathbb{R}$ together with a filtration $F$ on $\Aa\otimes_\RR\CC$. We have:

\begin{prop}\label{obstruction}Assume that a compact almost complex 4-manifold satisfies $E_1=E_\infty$ and $h^{0,1}=h^{1,0}$.
Then the manifold is formal.
\end{prop}
\begin{proof}
Proposition \ref{Hodgedec} gives Hodge decompositions for $H^1$ and $H^2$, whenever $h^{0,1}=h^{1,0}$ and by Serre duality we also have a Hodge decomposition on $H^3$.
Together with the degeneration condition at $E_1$, this makes the de Rham algebra with the Hodge filtration,
into a multiplicative pure Hodge complex. One may then apply the principle of two types of \cite{DGMS} (see also the purity implies formality proof in \cite{CiHo}).
\end{proof}

By Theorem \ref{complexfinale}, the hypotheses of the above proposition are always satisfied 
for a compact complex surface with even first Betti number:
\begin{coro}\label{coroformal}
 Any compact complex surface with $b_1$ even is formal.
\end{coro}

\begin{exam}We list some examples where the above result applies:
\begin{enumerate}

\item Let $\mathbf{E}^3$ be a principal circle bundle over a torus $\mathbb{T}^2$, and $\mathbf{E}^4$ a principal circle bundle over $\mathbf{E}^3$. In \cite{FGG} it is shown that whenever $b^2(\mathbf{E}^4)=2$, such manifolds do not admit integrable structures, although they are symplectic. The proof in \cite{FGG} uses Kodaira's classification, while Corollary \ref{coroformal} recovers the same result.

 \item More generally, any compact 4-dimensional nilmanifold is known to be either a torus or non-formal \cite{Ha89}. So, other than a torus, a non-formal 4-dimensional nilmanifold with $b_1$ even does not admit any integrable almost complex structure.

\item A large number of parallelizable four manifolds having no complex structure are presented in \cite{Bro}. The argument uses Kodaira's classification together with the existence of non-trivial Massey products.
Corollary \ref{coroformal} reproves these results without invoking Kodaira's classification.

\end{enumerate}
\end{exam}

Note that in the case of non-integrable structures on a compact almost complex 4-manifold, by \cite{CPS} we always have $E_1\neq E_\infty$, and so the hypotheses of the above proposition imply as well that the manifold is a complex surface.
We end this section with the following observation:

\begin{rema}
The analytic invariant $\Ii:=\dim E_1^{0,1}-\dim E_1^{1,0}$
takes the values $0$, $1$ or $\infty$
according to K\"ahler surfaces, complex non-K\"ahler surfaces or non-integrable almost complex 4-manifolds respectively.
\end{rema}

\section{Examples}\label{SecExamples}

A large source of examples of almost complex manifolds arises in the setting of nilmanifolds, defined by means of an even dimensional real Lie algebra $\mathfrak{g}$ together with an endomorphism $J:\mathfrak{g}\to \mathfrak{g}$ such that $J^2=-1$. When the coefficients of $\mathfrak{g}$ are rational, this data gives a compact almost complex manifold $M=G/\Gamma$ where $Lie(G)=\mathfrak{g}$ and $\Gamma$ is a cocompact subgroup.
Moreover, there the algebra of left-invariant differential forms $\Aa_\mathfrak{g}\hookrightarrow \Aa_{\dR}\otimes \CC$ injects into the de Rham algebra, preserves the bidegrees and computes the de Rham cohomology of the manifold.
The inclusion $\Aa^{*,*}_\mathfrak{g}\hookrightarrow \Aa^{*,*}$ of left-invariant forms into all forms 
and by Nomizu's Theorem there are bidegree-preserving isomorphisms in cohomology
\[\bigoplus {}^LE_\infty^{p,q}\cong \bigoplus E_\infty^{p,q}.\]
Note, that the left-invariant  spectral sequence may degenerate at an earlier stage than the true spectral sequence.
Still, the above isomorphism gives:

\begin{lemm}
For a left-invariant almost complex structure on a nilmanifold, the irregularity $\mathfrak{q}$ and hence all Hodge-de Rham numbers may be computed using left-invariant forms only. 
\end{lemm}

\subsection{Filiform nilmanifold}\label{fili}
Consider the real Lie algebra $\mathfrak{g}$ generated by $\{X_1, X_2, X_3, X_4\}$
with the only non-trivial Lie brackets given by 
\[
[X_1 ,X_2] = X_{3} \quad \textrm{and} \quad [X_1 ,X_3] = X_{4}.
\]
 Its Betti numbers are $b^1=b^2=2$ and so its associated compact nilmanifold does not admit any integrable structure by Corollary \ref{coroformal}. We have that  
 $\mathfrak{q}\in \{1,2\}$ for any almost complex structure on this nilmanifold, which
 give two Hodge-de Rham diamond types:
 \[
 \begin{tabular}{cccc}
\xymatrix@C=.1pc@R=.1pc{
 &&1\\
 &1&&1\\
 0&&2&&0\\
 &1&&1\\
 &&1
}&&& \xymatrix@C=.1pc@R=.1pc{
 &&1\\
 &0&&2\\
 0&&2&&0\\
 &2&&0\\
 &&1
}\\
\text{\small{Type 1}}&&& \text{\small{Type 2}}
\end{tabular}
\]

All related invariants are listed in the table below:
\begin{center}
\setlength\extrarowheight{5pt}
\begin{tabular}{lccccc}
 &$\mathfrak{q}$&$\widetilde\sigma$&$\sigma$&$\chi$&$p_g$\\\hline 
Type 1 (non-integrable)&1&0&0&0&0\\
Type 2 (non-integrable)&2&4&0&-1&0\\
\\
 \end{tabular}
 \end{center}
 
Note that for Type 1 we do have $\sigma=\widetilde\sigma$, 
but for Type 2 we only have 
$\sigma\equiv \widetilde\sigma\,(\mathrm{mod}\,4)$. Also, by Corollary \ref{cntsmalldef}, Hodge-de Rham diamonds of
Type 2 are constant under small deformations. Note that Type 1 diamond looks like the diamond of a compact K\"{a}hler surface. Its 
non-K\"{a}hlerness (and hence non-integrability) is uncovered by the failure of degeneration at $E_1$, which in turn is forced by Proposition \ref{obstruction}.

The above Hodge diamonds may be realized by left-invariant almost complex structures.
For instance, the almost complex structures $J_1$ and $J_2$ defined on $\mathfrak{g}$ by 
 \[J_1(X_1) = X_2,\, J_1(X_3)=X_4 \text{ and }J_2(X_1) = X_4,\, J_2(X_2)=X_3\]
 give Hodge-de Rham diamonds of Type 1 and 2 respectively.

Note that the left-invariant Fr\"{o}licher spectral sequence of any left-invariant almost complex structure of Type 1
will only degenerate at the second stage, in concordance with Proposition \ref{obstruction},
so ${}^LE_1\neq{}^LE_2=E_2=E_\infty$.
In contrast, Type 2 structures will always satisfy ${}^LE_1={}^LE_2=E_2=E_\infty$. In both cases, we have $^{L}E_1\neq E_1\neq E_\infty$.

\subsection{Kodaira-Thurston manifold} Consider the Kodaira-Thurston manifold 
\[
KT:=H_\Z \times \Z \backslash H \times \R
\]
where $H$ is the $3$-dimensional Heisenberg Lie group, $H_\Z$ is the integral subgroup, and the action is on the left. Its Betti numbers are $b^1=3$ and $b^2=4$ and so in this case we have $\mathfrak{q}\in \{2,3\}$, which give the Hodge-de Rham diamonds
\[\begin{tabular}{cccccc}
\xymatrix@C=.1pc@R=.1pc{
 &&1\\
 &1&&2\\
 1&&2&&1\\
 &2&&1\\
 &&1
}
&&
\xymatrix@C=.1pc@R=.1pc{
 &&1\\
 &1&&2\\
 0&&4&&0\\
 &2&&1\\
 &&1
}
&&
\xymatrix@C=.1pc@R=.1pc{
 &&1\\
 &0&&3\\
 0&&4&&0\\
 &3&&0\\
 &&1
}\\
\text{\small{Type 1}}&& \text{\small{Type 2}}&& \text{\small{Type 3}}
\end{tabular}
\]
and the associated invariants
\begin{center}
\setlength\extrarowheight{5pt}
\begin{tabular}{lccccc}
 &$\mathfrak{q}$&$\widetilde\sigma$&$\sigma$&$\chi$&$p_g$\\\hline 
Type 1 (integrable)&2&0&0&0&1\\
Type 2 (non-integrable)&2&-4&0&-1&0\\
Type 3 (non-integrable)&3&-8&0&-2&0\\
\\
 \end{tabular}
 \end{center} 
 
The Lie algebra of $H \times \R$ is spanned by $X,Y,Z,W$ with bracket $[X,Y] = -Z$.
The almost complex structures $J_1$ and $J_2$ defined on $\mathfrak{g}$ defined by 
\[J_1(X) = Y,\, J_1(Z) = W \text{ and } J_2(W) = X,\, J_2(Z) = Y\]
give left-invariant almost complex structures on $KT$ of Type 1 and 2 respectively.
Type 3 left-invariant structures are forbidden by dimensional reasons, since $\dim \Aa_{\mathfrak{g}}^{0,1}=2$.
In particular, Hodge-de Rham diamonds of Type 2 left-invariant structures are constant under small deformations.
Note that
Type 1 is integrable and so ${}^LE_1=E_1=E_\infty$, while Type 2 satisfies 
${}^LE_0={}^LE_1={}^LE_2=E_2=E_\infty$ and $E_1\neq E_2$.

\bibliographystyle{alpha}

\bibliography{biblio}

\end{document}